\documentclass{article}
\usepackage{amssymb}
\usepackage{amsmath}
\usepackage{mathrsfs}
\usepackage{graphicx}
\usepackage{subeqnarray}
\usepackage{subfig}
\usepackage{mathrsfs}
\usepackage{amsthm}
\usepackage{float}

\usepackage[pdftex,colorlinks=true,linkcolor=black,citecolor=black,bookmarksnumbered=true,urlcolor=blue]{hyperref}

\theoremstyle{definition}
\newtheorem{defn}{Definition}
\theoremstyle{plain}
\newtheorem{thm}{Theorem}
\newtheorem{cor}{Corollary}
\newtheorem{lemm}{Lemma}

\newcommand{\nats}{\mathbb{N}}

\newcommand{\reals}{\mathbb{R}}

\newcommand{\cfuncs}[1]{\mathscr{C}({#1})}

\newcommand{\boehms}[1]{\mathscr{B}({#1})}
\newcommand{\oboehms}[1]{\mathscr{B}_0({#1})}

\newcommand{\vare}{\varepsilon}

\newcommand{\rn}{\reals^N}
\newcommand{\toto}{\rightrightarrows}

\newcommand{\supp}{\mathrm{supp}}

\newcommand{\Knorm}[2]{\left\|{#1}\right\|_{#2}}

\newcommand{\Dto}{\overset{\Delta}\longrightarrow}
\newcommand{\dto}{\overset{\delta}\longrightarrow}
\renewcommand{\phi}{\varphi}

\begin{document}

\title{A Sheaf of Boehmians}
\author{Jonathan Beardsley\footnote{The first author is partially supported by NSF Grant DMS 0649159.}\\jon.s.beardsley@gmail.com\\Piotr Mikusi\'nski\\piotrm@mail.ucf.edu}
\maketitle{}

\textbf{Abstract:} We show that Boehmians defined over open sets of $\rn$ constitute a sheaf.  In particular, it is shown that such Boehmians satisfy the gluing property of sheaves over topological spaces.

\bigskip



\section{INTRODUCTION} \label{Ch1}

\indent The name {\it Boehmians} is used to describe a space of objects that are defined as equivalence classes of pairs of sequences $(x_n,\phi_n)$ where $x_n\in X$, some nonempty set, and $\phi_n\in G$, a commutative semigroup acting on $X$. If $X$ is equipped with a topology or sequential convergence, then it is usually assumed that $\phi_n x \to x$ as $n\to\infty$. In applications to generalized functions, $X$ is usually a space of functions and $G$ is a semigroup of ``test" functions acting on $X$ by convolution.

In the case of the basic space of Boehmians \cite{conv}, $X$ is the space of continuous functions on $\rn$ and $G$ is the convolution semigroup of smooth functions with compact support.  These Boehmians are basically global objects.  Boehmians on open subsets of $\rn$ were introduced in \cite{open}. The relationship between Boehmians on $\rn$ and Boehmians on open subsets of $\rn$ is not completely understood. Some related questions are discussed in \cite{flex}.

In this note we show that Boehmians on $\rn$ constitute a sheaf, a mathematical structure used to organize local information over open sets in a topological space. It is our hope that this will lead to a better understanding of the relationship between Boehmians on $\rn$ and Boehmians on open subsets of $\rn$. The structure of a sheaf should enable us to use new tools to attack some unsolved fundamental problems concerning Boehmians in $\rn$ discussed in \cite{flex}.
This work was inspired by discussions with Joseph Brennan.

In the next section we introduce some notation and definitions that the reader must be familiar with to define the Boehmians over $\rn$.  We continue with some lemmas and theorems that prove a gluing property for Boehmians over two open sets.  In the last section we extend this gluing property to the gluing property of sheaves, and show that the Boehmians do in fact constitute a sheaf.

\section{PRELIMINARIES} \label{Ch2}

If $K$ is a compact subset of an open set $U\subseteq\rn$, we write $K\Subset U$. We denote by $B_\vare(x)$ the ball of radius $\vare$ centered at $x$ and by $B_\vare$ the ball centered at the origin. For $U\subseteq\rn$, we denote
$$
U^{+\vare} =U+B_\vare=\bigcup_{x\in U} B_\vare(x)\quad \text{and} \quad U^{-\vare}=\left(\overline{U^\complement+B_\vare}\right)^\complement,
$$
where $A^\complement$ denotes the complement of $A$.

It is possible that $U^{-\vare}$ is empty even if $U$ is not empty. While this does not create any real problems, we will always implicitly assume that $\vare$ is such that $U^{-\vare}$ is not empty.

  We denote by $\cfuncs{U}$ the space of all continuous functions on some $U\subseteq\reals^N$ and define $\supp(f)=\overline{\{x\in\rn :f(x)\neq 0\}}$, where $\overline{S}$ denotes the closure of the set $S\subseteq\reals^N$. $\mathscr{D}(\rn)$, or simply $\mathscr{D}$, denotes the space of test functions, that is, smooth functions with compact support.  By $\mathscr{D}_o(\rn)$, or simply $\mathscr{D}_o$, we mean the collection of all $\varphi\in\mathscr{D}(\rn)$ such that $\varphi\geq 0$ and $\int\varphi=1$. For $\phi \in \mathscr{D}$ we define
$$
s(\phi)=\inf\{\vare>0:\supp(\phi)\subset B_\vare\}.
$$
A sequence $\phi_1, \phi_2, \dots \in \mathscr{D}_o$ is called a delta sequence if $s(\phi_n)\to 0$ as $n\to\infty$.  We denote the set of all delta sequences by $\Delta$.

For any $A,B\subseteq \rn$ we denote $d(A,B)=\inf\{\|x-y\|:x\in A, y\in B\}$, where $\|x-y\|$ is the Euclidean distance between $x$ and $y$.

If $f$ is a function defined on $U\subseteq\rn$, we will define $\Knorm{f}{U}$ to be the essential supremum of $f$ on $U$, or $\Knorm{f\vert_U}{\infty}$. We will also use the $L^1$-norm, denoted by $\Knorm{f}{1}=\int_{\rn} |f|$.

The convolution of $f\in \mathscr{C}(\rn)$ with a test function $\varphi\in \mathscr{D}(\rn)$ will be indicated by $f\ast\varphi$ and defined as
$$
(f\ast\varphi)(x) =\int_{\rn} f(y)\varphi(x-y)dy=\int_{\rn} f(x-y)\varphi(y)dy.
$$
In the case $f$ is a function defined on an open subset $U\subseteq\rn$, we adopt the convention that $f(s)\varphi(s)=0$ whenever $\varphi(s)=0$, even if $f$ is not defined at $s$. Note that, if $f\in\cfuncs{U}$ and $s(\varphi)=\vare$, then $f\ast\varphi$ is defined on $U^{-\vare}$.

If a sequence of functions $f_n$ is convergent to $f$ uniformly on a set $S$, we will write $f_n\toto f$ on $S$.

\begin{lemm}\label{normsplit}
If $K\Subset\rn$, $f\in\cfuncs{K}$, $\varphi\in\mathscr{D}$, and $s(\varphi)<\vare$, then
$$
\Knorm{f\ast\varphi}{K^{-\vare}}\leq\Knorm{f}{K}\|\varphi \|_1.
$$
\end{lemm}

\begin{cor}\label{cc}
Let $\phi \in \mathscr{D}_o$ and $s(\varphi)<\vare$ for some $\vare >0$. If a sequence of continuous functions $f_n \toto 0$ on $K\Subset\rn$, then $f_n\ast\varphi \toto 0$ on $K^{-\vare}$.
\end{cor}

\begin{thm}\label{infiniteconv}
If $(\phi_n)\in\Delta$ and $\sum_{n=1}^{\infty}s(\phi_n)<\infty$, then the sequence
$$
{\psi_n=\phi_1\ast\phi_2\ast\ldots\ast\phi_n}
$$
 is convergent uniformly on $\reals^N$ to a function $\psi\in\mathscr{D}$ such that $s(\psi)\leq\sum_{n=1}^{\infty}s(\phi_n)$.  The limit of the sequence $\psi_n$ will be denoted by $\phi_1\ast\phi_2\ast\ldots$.
\end{thm}

\begin{proof}
See \cite{open}.
\end{proof}

If $U\subseteq\rn$, $K\Subset U$, $(\varphi_n)\in\Delta$, and $f\in\cfuncs{U}$, then $(f\ast\varphi_n)$ is a sequence of functions of which only a finite number of elements can be not defined on all of $K$. This situation arises often in what follows in this paper.  We are going to ignore this fact and say, for example, \textit{the sequence of functions $(f\ast\varphi_n)$ converges uniformly on $K$} instead of saying something like \textit{there exists an $n_0\in\nats$ such that  the sequence of functions $(f\ast\varphi_n)$, for $n>n_0$, converges uniformly on $K$}. The following lemma is an example of such a situation.

\begin{lemm}\label{samelimit}
Let $U\subseteq\rn$ and $K\Subset U$. If $f\in\cfuncs{U}$ and $(\varphi_n)\in\Delta$, then ${\Knorm{f-f\ast\varphi_n}{K}\to0}$.
\end{lemm}

\begin{proof} Let $\vare > 0$ be such that $\overline{K^{+\vare}}\subseteq U$ and let $\eta >0$. Since $\overline{K^{+\vare}}$ is compact, $f$ is uniformly continuous on $\overline{K^{+\vare}}$  and there is a $\delta>0$ such that  $\vert f(x)-f(y)\vert<\eta$ whenever $x,y\in\overline{K^{+\vare}}$ and $\| x-y\|<\delta$. Without loss of generality, we can assume that $\delta < \vare$. Let $n_0$ be such that $s(\varphi_n)<\delta$ for all $n>n_0$. Note that, for $n > n_0$, $x\in K$ and $y\in B_{s(\varphi_n)}$ ensures that $x-y\in \overline{K^{+\vare}}$ and $\vert x-(x-y)\vert<\delta$. Thus, for $n>n_0$, we have
\begin{align*}
\Knorm{f-f\ast\varphi_n}{K}&=\sup\limits_{x\in K}\left| f(x)-\int\limits_{B_{s(\varphi_n)}}f(x-y)\varphi_n(y)dy\right|\\
&=\sup\limits_{x\in K}\left| f(x)\int\limits_{B_{s(\varphi_n)}}\varphi_n(y)dy-\int\limits_{B_{s(\varphi_n)}}f(x-y)\varphi_n(y)dy\right|\\
&=\sup\limits_{x\in K}\left| \int\limits_{B_{s(\varphi_n)}}(f(x)-f(x-y))\varphi_n(y)dy\right|\\
&\leq\sup\limits_{x\in K}\int\limits_{B_{s(\varphi_n)}}\left|f(x)-f(x-y)\right|\left|\varphi_n(y)\right|dy\\
&<\eta\Knorm{\varphi_n}{1}=\eta.
\end{align*}
\end{proof}

\begin{cor}\label{seqconv}
Let $U$ be an open subset of $\rn$ and let $(f_n)$ be a uniformly convergent sequence of functions on $K\Subset U$. If  $(\varphi_n)\in\Delta$ and $s(\varphi_n)<\vare$ for all $n$, then ${(f_n - f_n\ast\varphi_n)\toto 0}$  on $K^{-\vare}$.
\end{cor}

\begin{proof} If $f_n \toto f$ on $K$, then
\begin{align*}
\Knorm{f_n\ast\varphi_n-f_n}{K^{-\vare}}&\leq\Knorm{f_n\ast\varphi_n-f\ast\varphi_n}{K^{-\vare}}+\Knorm{f\ast\varphi_n-f_n}{K^{-\vare}}\\
&\leq \Knorm{(f_n-f)\ast\phi_n}{K^{-\vare}}+\Knorm{f\ast\varphi_n-f}{K^{-\vare}}+\Knorm{f-f_n}{K^{-\vare}}\\
&\leq\Knorm{f_n-f}{K}+\Knorm{f\ast\varphi_n-f}{K^{-\vare}}+\Knorm{f-f_n}{K^{-\vare}} \to 0,
\end{align*}
by Lemma \ref{normsplit} and Lemma \ref{samelimit}.
\end{proof}

\section{BOEHMIANS ON OPEN SETS} \label{Ch3}

\indent In this section we describe the construction of Boehmians on open subsets of $\rn$.  They are defined as equivalence classes of fundamental sequences of continuous functions.

\begin{defn}
Let $U\subseteq\rn$ be an open set and let $f_n\in\cfuncs{U}$, for $n\in\nats$. We say that the sequence $(f_n)$ is a \textit{fundamental sequence} if for every $K\Subset U$ there is $(\delta_n)\in\Delta$ such that, for every $m\in\nats$, the sequence $(f_n\ast\delta_m)$ is uniformly convergent on $K$ as $n\to\infty$ . The set of all fundamental sequences on $U$ will be denoted by $\mathscr{A}(U)$.\end{defn}

If $(f_n),(g_n)\in \mathscr{A}(U)$, we write $(f_n)\sim(g_n)$ when for every $K\Subset U$ there is $(\varphi_n)\in\Delta$ such that, for each $m$, $(f_n-g_n)\ast\varphi_m\toto 0$ on $K$ as $n\to\infty$.

\begin{lemm}\label{equiv}
$\sim$ is an equivalence relation.
\end{lemm}
\begin{proof}
We will show that $\sim$ is transitive. Let $(f_n),(g_n),(h_n)\in\mathscr{A}(U)$ and let $K\Subset U$. There are $(\varphi_n),(\psi_n)\in\Delta$ such that $(f_n-g_n)\ast\varphi_m\toto 0$ and $(g_n-h_n)\ast\psi_m\toto 0$  on $K$ as $n\to\infty$ for every $m$.  Let $\delta_n=\varphi_n\ast\psi_n$.  Then $(\delta_n)\in\Delta$ and, for any $m$,
$$
(f_n-h_n)\ast\delta_m=(f_n-g_n)*\phi_m*\psi_m+(g_n-h_n)*\psi_m*\phi_m \toto 0
$$
on $K$,  by Corollary \ref{cc}. Therefore, $(f_n)\sim (h_n)$.
\end{proof}

By the Boehmians on $U$ we mean the equivalence classes of $\mathscr{A}(U)$ with respect to $\sim$.  The equivalence class of $(f_n)\in\mathscr{A}(U)$ will be denoted by $[(f_n)]$.  The space of all Boehmians on $U$ will be denoted by $\boehms{U}$, that is, $\boehms{U}=\mathscr{A}(U)/\sim$.

Note that for any function $f\in\cfuncs{U}$, the constant sequence $(f)$ is fundamental in $U$.  As such, there is a well defined inclusion $\boehms{U}\to\cfuncs{U}$ where a continuous function $f$ is mapped to the equivalence class $[(f)]\in\boehms{U}$.  Suppose we have a Boehmian $F\in\boehms{U}$ and a continuous function $f\in\cfuncs{U}$.  If $(f)\sim(f_n)$ for every $(f_n)\in F$, then we will write that $F=f$ on $U$ or $F\in\cfuncs{U}$ (though this is a slight abuse of notation). Similarly, we will write $\cfuncs{U}\subset\boehms{U}$.

While a Boehmian on a closed set is not defined, we will sometimes write $F\in\cfuncs{K}$ for $K$ a compact set.  In this case we mean that there is $U\supset{K}$ and $f\in\cfuncs{U}$ such that $F=f$ on $U$.

\begin{defn}
For $[(f_n)],[(g_n)]\in\boehms{U}$ and $r\in\reals$ we define:

\end{defn}

\begin{thm}
$\boehms{U}$ is a vector space over $\reals$ with the addition and scalar multiplication defined as
$$
[(f_n)]+[(g_n)]=[(f_n+g_n)] \quad \text{and} \quad r[(f_n)]=[(rf_n)].
$$
\end{thm}

\begin{proof}
We show that addition and scalar multiplication of Boehmians are well defined. For $(f_n),(f'_n)\in [(f_n)]$ and $(g_n),(g'_n)\in[(g_n)]$ and any $K\Subset U$, we have $(\phi_n),(\psi_n)\in\Delta$ such that $(f_n-f'_n)\ast\phi_m \toto 0$ and $(g_n-g'_n)\ast\psi_m \toto 0$ on $K$ for each $m$.  Choosing $m'$ large enough (equivalently, $s(\phi_{m'})$ small enough) such that $((f_n+g_n)-(f'_n+g'_n))\ast(\phi_m\ast\psi_m)$ is defined on $K$ for every $m\geq m'$, we construct the delta sequence $(\phi_{m'+n}\ast\psi_{m'+n})$ such that
\begin{align*}
&\Knorm{((f_n+g_n)-(f'_n+g'_n))\ast(\phi_{m'+k}\ast\psi_{m'+k})}{K}\\
&\quad =\Knorm{(f_n-f'_n+g_n-g'_n)\ast(\phi_{m'+k}\ast\psi_{m'+k})}{K}\\
&\quad \leq\Knorm{(f_n-f'_n)\ast(\phi_{m'+k}\ast\psi_{m'+k})}{K}+\Knorm{(g_n-g'_n)\ast(\phi_{m'+k}\ast\psi_{m'+k})}{K},
\end{align*}
which converges to zero for every $k$ by Corollary \ref{cc}.  Thus ${(f_n+g_n)\sim (f'_n+g'_n)}$, so the addition is well defined. For scalar multiplication, it is sufficient to note that $r(f_n\ast\phi_m)=(rf_n\ast\phi_m)$.
\end{proof}

\section{A GLUING PROPERTY} \label{Ch4}

\indent By the restriction of a Boehmian $F\in\boehms{U}$ to an open set $V\subseteq U$, we mean $F\vert_V=[(f_n\vert_V)]$.  It is clear that this is well defined.

In this section we show that if $U,V\subset\rn$ are open, $U\cap V\neq\emptyset$, $F\in\boehms{U}$, $G\in\boehms{V}$, and $F\vert_{U\cap V}=G\vert_{U\cap V}$, then there is a Boehmian $H\in\boehms{U\cup V}$ such that $H\vert_U=F$ and $H\vert_V=G$. First we prove several technical lemmas.

\begin{lemm}\label{convdef}
Let $U\subseteq \rn$ be an open set, $\vare >0$, $\varphi\in\mathscr{D}_o$, and $s(\varphi)<\vare$.
\begin{enumerate}
\item[(i)] If $(f_n)\in \mathscr{A}(U)$, then $(f_n\ast\varphi)\in \mathscr{A}(U^{-\vare})$.
\item[(ii)] If $(f_n) \sim (g_n)$ in $\mathscr{A}(U)$, then $(f_n\ast\varphi) \sim (g_n \ast\varphi)$ in $\mathscr{A}(U^{-\vare})$.
\item[(iii)] If $[(f_m)]\in\boehms{U}$, then $[(f_m\ast\varphi)]\in \boehms{U^{-\vare}}$.
\end{enumerate}
\end{lemm}

\begin{proof}
Let $(f_n)\in \mathscr{A}(U)$ and let $K\Subset U^{-\vare}$. Then  $\overline{K^{+\vare}}\Subset U$, so there exists $(\delta_n)\in\Delta$ such that  $(f_n\ast\delta_m)$ converges uniformly on $\overline{K^{+\vare}}$ as $n\to\infty$, for every $m\in\nats$. Since $s(\varphi)<\vare$, $(f_n\ast\delta_m\ast\varphi)$ is well defined on $K$ for  all $m,n\in \nats$.  Moreover, by Lemma \ref{cc}, for every $m\in\nats$ the sequence $(f_n\ast\delta_m)\ast\varphi=(f_n\ast\varphi)\ast\delta_m$ is uniformly convergent on $K$.  Hence $(f_n\ast\varphi)\in \mathscr{A}(U^{-\vare})$, so (i) is shown.

If $(f_n)\sim(g_n)$ in $\mathscr{A}(U)$, then there is a $(\delta_n)\in\Delta$ such that, for each $m\in\nats$, $(f_n-g_n)\ast\delta_m\toto 0$ on $\overline{K^{+\vare}}$ as $n\to\infty$.  Hence, for every $m\in\nats$, $(f_n\ast\varphi-g_n\ast\varphi)\ast\delta_m=((f_n-g_n)\ast\delta_m)\ast\varphi\toto 0$ on $K$ as $n\to\infty$, again by Lemma \ref{cc}. Consequently, $(f_n*\varphi) \sim (g_n *\varphi)$ in $\mathscr{A}(U^{-\vare})$, which shows (ii).

Item (iii) is a direct consequence of (i) and (ii).
\end{proof}

\begin{cor}\label{welldef}
Let $U\subseteq\rn$ be open, $F\in\boehms{U}$, $\vare>0$, $\varphi\in\mathscr{D}_o$, and $s(\varphi)<\vare$. If $(f_n)\in F$, then $F\ast\varphi=[(f_n\ast\varphi)]$ is well defined as a Boehmian on $U^{-\vare}$.
\end{cor}

\begin{lemm}\label{funccor}
Let $U\subseteq\rn$. If $F=[(f_n)]\in\boehms{U}$ and $K\Subset U$, then there exists a $\psi\in\mathscr{D}_o$ such that $F\ast\psi\in\mathscr{C}(K)$ and $(f_n\ast\psi)\toto F\ast\psi$ on $K$.
\end{lemm}

\begin{proof}
 Let $K\Subset U$.  We show that there is $\psi\in\mathscr{D}_o$ such that $(f_n\ast\psi)$ converges uniformly on $K$ and that its limit can be identified with $F\ast\psi$ on an open set containing $K$.  Take an $\vare >0$ such that $K_1=\overline{K^{+\vare}} \Subset U$.  Since $(f_n)$ is fundamental, there is $(\phi_n)\in\Delta$ such that, for each $m\in\nats$, $(f_n\ast\phi_m)$ converges uniformly on $K_1$ as $n\to\infty$.  Without loss of generality, we can assume that $s(\phi_n) < \vare / 2$ for all $n$. Let $\psi=\phi_1$ and let $g$ be the uniform limit of $(f_n\ast\psi)$ on $K_1$.  By Corollary \ref{cc}, for every $m$, $(f_n\ast\psi - g)\ast\phi_m \toto 0$ on $K^{+\frac{\vare}2}$. Hence, $(f_n\ast\psi)\sim(g)$ on $K^{+\frac{\vare}2}$. Consequently, $F\ast\psi=[g]$ as elements of $\boehms{K^{+\frac{\vare}2}}$ and we have $F\ast\psi=g$ on $K$.
\end{proof}

\begin{lemm}\label{commute}
Let $U\subseteq\rn$ be open, $F\in\boehms{U}$, $\vare>0$, $\varphi\in\mathscr{D}_o$, and $s(\varphi)<\vare$.  Then $(F\ast\phi)\vert_{U^{-\vare}}=F\vert_U\ast\phi$ on $U^{-\vare}$.
\end{lemm}

\begin{lemm}\label{findseq}
Let $U\subseteq\rn$ be open and let $F\in\boehms{U}$. For every $\vare>0$ there is a $\varphi\in\mathscr{D}_o$ such that $F\ast\varphi\in\mathscr{C}(U^{-\vare})$.
\end{lemm}
\begin{proof}
If $U^{-\vare}$ is bounded, then we can choose $K$ such that $U^{-\vare}\subset K\Subset U$ and then use Lemma \ref{funccor}. If $U^{-\vare}$ is unbounded, then there exists a sequence of open sets $U_n$ such that $U^{-\vare} \subseteq \bigcup_{n=1}^\infty U_n\subseteq U$,
$\overline{U_n}\Subset U$ for every $n\in\mathbb{N}$, and
for each $x\in U^{-\vare}$ there is an $n\in\mathbb{N}$ such that $B_{\frac{\vare}{2}}(x)\subseteq U_n$. By Lemma \ref{funccor}, for every $n$ there is a $\varphi_n\in\mathscr{D}_o$ such that $s(\varphi_n)<\frac{\vare}{2^{n+1}}$ and $F\ast\varphi_n\in\cfuncs{\overline{U_n}}$.  Define $\varphi=\varphi_1\ast\varphi_2\ast\ldots$. Then $\varphi\in\mathscr{D}_o$ and $s(\varphi)<\sum\frac{\vare}{2^{n+1}}=\frac{\vare}{2}$, by Theorem \ref{infiniteconv}.  To show continuity of $F\ast\varphi$ in $U^{-\vare}$ consider an arbitrary $x\in U^{-\vare}$.  Then $x\in U_n$, for some $n$, and $F\ast\varphi_n$ is continuous at $x$.  Since $s(\varphi_1\ast\ldots\ast\varphi_{n-1}\ast\varphi_{n+1}\ast\ldots)<\frac{\vare}{2}$, the convolution $F\ast\varphi=(F\ast\varphi_n)\ast(\varphi_1\ast\ldots\ast\varphi_{n-1}\ast\varphi_{n+1}\ast\ldots)$ is well defined on $U_n$ and hence is continuous at $x$.
\end{proof}

\begin{lemm}\label{sameoninter}
Let $U$ and $V$ be open sets in $\rn$ with a non-empty intersection $W$.  If $F\in\boehms{U}$, $G\in\boehms{V}$, and $F\vert_W=G\vert_W$, then for any $\vare>0$ there exists $(\delta_n)\in\Delta$ such that the following two conditions hold for every $n\in\nats$:
\begin{itemize}
\item[(i)] $F\ast\delta_n\in\cfuncs{U^{-\vare}}$ and $G\ast\delta_n\in\cfuncs{V^{-\vare}}$,
\item[(ii)] $F\ast\delta_n=G\ast\delta_n$ on $W^{-\vare}=U^{-\vare}\cap V^{-\vare}$.
\end{itemize}
\end{lemm}

\begin{proof} (i) Let $\vare>0$. If $\vare_n=\frac{\vare}{2^n}$, then $\left(U^{-\vare_n}\right)$ is an increasing sequence of sets such that $\bigcup_{n\in\nats}U^{-\vare_n}=U$.  By Lemma \ref{findseq}, there exists a sequence $\varphi_n\in\mathscr{D}_o$ such that $F\ast\varphi_n\in\cfuncs{U^{-\vare_n}}$ and $s(\varphi_n)<\vare_n$.  Similarly, there exists a sequence $\psi_n\in\mathscr{D}_o$ such that $G\ast\psi_n\in\cfuncs{V^{-\vare_n}}$ and $s(\psi_n)<\vare_n$. Let $\delta_n=\varphi_n\ast\psi_n$. For each $n$, $F\ast\delta_n=F\ast\varphi_n\ast\psi_n$ is a continuous function on $U^{-\vare}$. Similarly, $G\ast\delta_n$ is a continuous function on $V^{-\vare}$.

(ii) Let $f_m^\ast=F\ast\delta_m$ and $g_m^\ast=G\ast\delta_m$.  By Corollary \ref{welldef}, for any $(f_n)\in F$ and $(g_n)\in G$, $(f_n\ast\delta_m)\sim\{f_m^\ast,f_m^\ast,f_m^\ast\ldots\}$ on $U^{-\vare}$ and $(g_n\ast\delta_m)\sim\{g_m^\ast,g_m^\ast,g_m^\ast\ldots\}$ on $V^{-\vare}$. We must show then that $(f_n\ast\delta_m)\sim (g_n\ast\delta_m)$ on $W^{-\vare}$ for each $m$.  Let $K\Subset W^{-\vare}$ and fix $m\in\nats$. Since $\overline{K^{+\vare}}\Subset W$ and $F=G$ on $W$, there is a sequence $(\omega_n)\in\Delta$ such that $\Knorm{((f_n-g_n)\ast\omega_k)}{\overline{K^{+\vare}}}\to 0$ for each $k$, as $n\to\infty$.  Note, $(f_n\ast\delta_m-g_n\ast\delta_m)\ast\omega_k=((f_n-g_n)\ast\delta_m)\ast\omega_k$ is defined on $K$ since $(f_n-g_n)\ast\omega_k$ is defined on $\overline{K^{+\vare}}$ and  $\vare>s(\delta_1)\geq s(\delta_i)$ for all $i$. So $\Knorm{(f_n-g_n)\ast\delta_m\ast\omega_k}{K}\to 0$ by Lemma \ref{cc}.  So $(f_m^\ast)\sim(f_n\ast\delta_m)\sim(g_n\ast\delta_m)\sim(g_m^\ast)$.  Hence $(f_m^\ast)\sim(g_m^\ast)$ for each $m$.  However, both are constant sequences, so $f_m^\ast=g_m^\ast$ for every $m$.
\end{proof}

\begin{lemm}\label{simply}
Suppose $(f_n)$ is a fundamental sequence on an open set $U\subseteq\rn$ and $(g_n)$ is a sequence of continuous functions on the same set. Assume that for every $K\Subset U$ there exists $(\phi_n)\in\Delta$ such that for each $m\in\nats$, $(f_n-g_n)\ast\phi_m\toto 0$ on $K$ as $n\to\infty$. Then $(g_n)$ is fundamental on $U$ and $(g_n)\sim(f_n)$.
\end{lemm}

\begin{proof}
Choose $K\Subset U$.  As $(f_n)$ is fundamental, there exists $(\delta_n)\in\Delta$ such that $\Knorm{(f_n-f_m)\ast\delta_k}{K}\to 0$ for all $k\in\nats$ as $n,m\to\infty$.  Since there exists $(\varphi_n)\in\Delta$ such that  $\Knorm{(f_n-g_n)\ast\phi_k}{K}\to 0$ for each $k$, for any $k$ we have the following:
\begin{align*}
\Knorm{(g_n-g_m)\ast\varphi_k\ast\delta_k}{K}\leq &\Knorm{(g_n-f_n)\ast\varphi_k\ast\delta_k}{K}+\Knorm{(f_n-f_m)\ast\varphi_k\ast\delta_k}{K}\\
&\quad +\Knorm{(f_m-g_m)\ast\varphi_k\ast\delta_k}{K}.\\
\end{align*}
Each term of the above sum goes to zero by Lemma \ref{cc}, hence $(g_n)$ is fundamental on $U$. Clearly, $(g_n)\sim(f_n)$.
\end{proof}

\begin{lemm}\label{contex}
Let $U\subseteq\rn$ be open and let $(U_{n})$ be an increasing sequence of open sets such that
$\bigcup_{n\in\nats}U_{n}=U$ and $\overline{U_n} \Subset U_{n+1}$ for each $n\in \mathbb{N}$. If $F\in\boehms{U}$, then there exists $(\varphi_n)\in\Delta$ such that $F\ast\varphi_n\in\cfuncs{\overline{U_n}}$ for each $n\in \mathbb{N}$. If $(g_n)$ is any sequence of continuous extensions of $(F\ast\varphi_n)$  to $U$, then
$(g_n)$ is fundamental and $\left[(g_n)\right] = F$.
\end{lemm}

\begin{proof}
There exists $(\varphi_n)\in\Delta$ such that $F\ast\varphi_n\in\cfuncs{\overline{U_n}}$ for each $n\in \mathbb{N}$ by Lemma \ref{findseq}. Let $(g_n)$ be a sequence of continuous extensions of $(F\ast\varphi_n)$  to $U$ and let $F=[(f_n)]$.  If $K\Subset U$, then there is an $m\in\mathbb{N}$ such that $K\subset U_m$.  Note that $g_n=F*\varphi_n$ on $U_{m+1}$ for every $n>m$.  Moreover, there is $(\delta_n)\in\Delta$ and functions $h_k\in\cfuncs{\overline{U_{m}}}$ such that, for every $k$, $f_n\ast\delta_k \toto h_k$ on $\overline{U_{m}}$ as $n\to\infty$.  By Lemma \ref{funccor}, $h_k=F\ast\delta_k$ on $\overline{U_{m}}$. Let $\vare >0$ be such that $K^{+\vare} \subset U_{m}$. Now, for all $n>m$ and all $k$ such that $s(\delta_k)<\vare$ we have
\begin{align*}
\Knorm{(g_n-f_n)\ast\delta_k}{K}&=\Knorm{g_n\ast\delta_k-f_n\ast\delta_k}{K}\\
&\leq\Knorm{F\ast\delta_k\ast\varphi_n-h_k}{K}+\Knorm{h_k-f_n\ast\delta_k}{K}\\
&=\Knorm{h_k\ast\varphi_n-h_k}{K}+\Knorm{h_k-f_n\ast\delta_k}{K}\to 0.\\
\end{align*}
Hence, by Lemma \ref{simply}, $(g_n)$ is fundamental on $U$ and $(g_n)\sim (f_n)$.
\end{proof}

\begin{thm}\label{glue2}
Let $U$ and $V$ be open sets in $\rn$ with a non-empty intersection $W$.  If $F\in\boehms{U}$, $G\in\boehms{V}$, and $F=G$ on $W$, then there is a Boehmian $H\in\boehms{U\cup V}$ such that $H\vert_U=F$ and $H\vert_V=G$.
\end{thm}

\begin{proof}
  By Lemma \ref{sameoninter}, there exist $\vare_n\to 0$ and $(\delta_n)\in\Delta$ such that, for each $n\in\mathbb{N}$, we have $s(\delta_n)<\vare_n$, $F\ast\delta_n\in\cfuncs{U^{-\vare_{n+1}}}$, $G\ast\delta_n\in\cfuncs{V^{-\vare_{n+1}}}$, and $F\ast\delta_n=G\ast\delta_n$ on $U^{-\vare_{n+1}}\cap V^{-\vare_{n+1}}=W^{-\vare_{n+1}}$. Without loss of generality, we can assume that both $(\vare_n)$ and $s(\delta_n)$ are decreasing sequences.   Define a sequence of functions $(h_n)$ on $U^{-\vare_n} \cup V^{-\vare_n}$ by:
\begin{displaymath}
   h_n(x) = \left\{
     \begin{array}{lr}
       F\ast\delta_n & \text{if } x \in U^{-\vare_n}\\
       G\ast\delta_n & \text{if } x \in V^{-\vare_n}
     \end{array}
   \right.
\end{displaymath}
Let $(\tilde{h}_n)$ be the sequence of continuous extensions of $(h_n)$ from $U^{-\vare_n}\cup V^{-\vare_n}$ to $U\cup V$.
First we show that $(\tilde{h}_n)$ is a fundamental sequence on $U\cup V$.  Let $K\Subset U\cup V$ and let $n_0\in\nats$ be such that $K\Subset U^{-\vare_n}\cup V^{-\vare_n}$ for $n>n_0$.  Let $K_1$ and $K_2$ be compact subsets of $U\cup V$ such that $K_1\Subset U^{-\vare_n}$, $K_2\Subset V^{-\vare_n}$, and $K_1\cup K_2 = K$. For each $n>n_0$, let $f_n$ be a continuous extension of $F\ast \delta_n$ from $U^{-\vare_n}$ to  $U$ and let $g_n$ be a continuous extension of $G\ast \delta_n$ from $V^{-\vare_n}$ to $V$. Note that, by Lemma \ref{contex}, $(f_n)$ is a fundamental sequence on $U$ and $(g_n)$ is fundamental sequence on $V$. Let $\eta_1=d(K_1,U^\complement)$ and $\eta_2=d(K_2,V^\complement)$.    There exists a sequence $(\varphi_n)\in\Delta$ such that, for each $m$, $(f_n\ast\varphi_m)$ is uniformly convergent on $\overline{K_1+B_{\frac{\eta_1}{2}}}$ as $n\to\infty$.  Similarly, there exists $(\psi_n)\in\Delta$ such that, for each $m$, $(g_n\ast\psi_m)$ is uniformly convergent on $\overline{K_2+B_{\frac{\eta_2}{2}}}$.  Then, by Lemma \ref{cc}, the sequences $(f_n\ast\varphi_m\ast\psi_m)$ and $(g_n\ast\psi_m\ast\varphi_m)$  are uniformly convergent on $K_1$ and $K_2$, respectively.  Since, for sufficiently large $m$, $\tilde{h}_n\ast\varphi_m\ast\psi_m=f_n\ast \varphi_m\ast\psi_m$ on $K_1$ and $\tilde{h}_n\ast\psi_m\ast\varphi_m=g_n \ast \psi_m\ast\varphi_m$ on $K_2$, $\tilde{h}_n\ast \psi_m\ast\varphi_m$ is uniformly convergent on $K$.  Conseqently,  $(\tilde{h}_n)$ is a fundamental sequence on $U\cup V$.  Let $H=[(\tilde{h}_n)]$.

If $K\Subset U$, then there exists $n_1$ such that $K\subset U^{-\vare_n}\cup V^{-\vare_n}$ for all $n>n_1$.  Then, for $n>n_1$, $\tilde{h}_n=f_n$ on $K$.  By Lemma 10, $[(\tilde{h}_n\vert_U)]=F$. Similarly, $H\vert_V=G$.
\end{proof}

\begin{cor}
Suppose $\{U_i\}_{i=1}^{n}$ is a finite collection of open sets in $\rn$ and $\{F_i\}_{i=1}^{n}$ is a collection of Boehmians such that $F_i\in\boehms{U_i}$ and $F_i\vert_{U_i\cap U_j}=F_j\vert_{U_i\cap U_j}$.  If we define $U=\cup_{i=1}^{n} U_i$ then there is a Boehmian $F\in\boehms{U}$ such that $F\vert_{U_i}=F_i$ for $i=1, \dots , n$.
\end{cor}

\section{A SHEAF OF BOEHMIANS} \label{Ch5}

\indent With Theorem \ref{glue2} in hand, it can be shown that Boehmians are a sheaf.  In general, a sheaf is a way of organizing data over open sets of a topological space.  It is a construction which has found useful applications in complex variables, algebraic geometry, and algebraic and differential topology \cite{sheaf}.  To say that some set can be considered a sheaf over a topological space, it must first be a presheaf.  That is to say that the set must have a well defined idea of local behavior and restriction.  For the space of Boehmians $\boehms{\rn}$, it has been done in \cite{open}.  More explicitly:

\begin{thm}[Presheaf]
The set of Boehmians $\boehms{\rn}$ constitutes a presheaf over $\rn$, that is
\begin{itemize}
\item[(i)] For each open $U\subseteq \rn$, there exists the space of Boehmians $\boehms{U}$ associated to $U$ and
\item[(ii)] For all open $U,V\subseteq \rn$, if $V\subseteq U$, then there there exist restriction maps $\rho_{V,U}:\boehms{U}\to \boehms{V}$ such that $\rho_{V,V}:\boehms{V}\to\boehms{V}$ is the identity map, and for open $W\subseteq V\subseteq U$, $\rho_{W,V}\circ\rho_{V,U}=\rho_{W,U}$.
\end{itemize}
\end{thm}

The reader familiar with category theory will notice that a presheaf is nothing more than a contravariant functor from the category of open sets over a space (with inclusion morphisms) to another appropriately chosen category where the morphisms are restriction maps.  If we have the presheaf property, then we can say that the Boehmians are a sheaf given certain other conditions.  The following are the conditions for being a sheaf, as explicated in \cite{sheaf}:
\begin{itemize}
\item[(a)] Let $\{U_\alpha\}$ be a family of open sets in $\rn$ and $U=\bigcup U_\alpha$.  If $F,G\in\boehms{U}$ are such that $F\vert_{U_\alpha}=G\vert_{U_\alpha}$ for each $\alpha$, then $F=G$.
\item[(b)] Let $\{U_\alpha\}$ be a family of open sets in $\rn$ and $U=\bigcup U_\alpha$. If there are Boehmians $F_\alpha\in\boehms{U_\alpha}$ such that $F_\alpha\vert_{U_\alpha\cap U_\beta}=F_\beta\vert_{U_\alpha\cap U_\beta}$ for every $\alpha,\beta$, then there exists $F\in\boehms{U}$ such that $F\vert_{U_\alpha}=F_\alpha$.
\end{itemize}

We will show that Boehmians satisfy these conditions. For (a) we will use the following lemma.
\begin{lemm}\label{prop1}
Let $\{U_\alpha\}$ be a family of open sets in $\rn$ and $U=\bigcup U_\alpha$.  If $(f_n)\sim(0)$ on $U_\alpha$ for each $\alpha$ then $(f_n)\sim(0)$ on $U$.
\end{lemm}

\begin{proof}
Let $K\Subset U$. Then there are $U_{\alpha_1},\ldots, U_{\alpha_i}$ such that $K\Subset U_{\alpha_1}\cup\ldots\cup U_{\alpha_i}$.  Moreover, there are $K_1,\ldots, K_i$ such that $K_j\Subset U_{\alpha_j}$ for $1\leq j\leq i$ and $K=K_1\cup\ldots\cup K_i$.  Let $\vare=\min\{d(K_j,U_j^\complement):1\leq j\leq i\}$.  By the assumption, for each $j$ there exists $(\varphi_n^j)\in \Delta$ such that $s(\phi_n^j)<\frac{\vare}{i}$ for all $n$, and for every $m$, $\Knorm{f_n\ast\varphi_m^j}{K_j}\to 0$ as $n\to\infty$.  Let $\varphi_n=\varphi_n^1\ast\ldots\ast\varphi_n^i$.  From \cite{open}, we have that $(\varphi_n)\in\Delta$ and $s(\phi_n)<\vare$ for each $n$.  And, for every $m$ and every $K_j$ we have that $\Knorm{f_n\ast\varphi_m}{K_j}\to 0$ since $f_n\ast\varphi_m=(f_n\ast\varphi_m^j)\ast(\varphi_m^1\ldots\varphi_m^{j-1}\ast\varphi_m^{j+1}\ldots\varphi_m^i)$. So $(f_n)\sim(0)$ on $U$.
\end{proof}

\begin{lemm}\label{countable}
Let $\{U_n\}$ be a countable family of open sets in $\rn$ and let $U=\bigcup U_n$. If $F_n\in\boehms{U_n}$ and $F_m\vert_{U_m\cap U_n}=F_n\vert_{U_m\cap U_n}$ for every $m,n\in\nats$, then there exists $F\in\boehms{U}$ such that $F\vert_{U_n}=F_n$ for all $n\in\nats$.
\end{lemm}

\begin{proof}
Define $V_k=\bigcup_{i=1}^{k}U_{i}$ and $G_1=F_1$. By Theorem \ref{glue2}, for all $k>1$ there exist Boehmians $G_k\in\boehms{V_K}$ such that $G_{k+1}\vert_{V_k}=G_{k}$.

Let $(\vare_n)$ be a decreasing sequence of real numbers such that $\vare_n\to 0$.  By Lemma \ref{findseq}, there exists a sequence $(\phi_n)\in\Delta$ such that $s(\phi_n)<\vare_n$ and $G_n\ast\phi_n\in\cfuncs{V_n^{-\vare_n}}$. Let, for each $n\in\nats$, $g_n$ be a continuous extension of $G_n\ast\phi_n$ to $U$. We claim that $(g_n)$ is fundamental on $U$ and $[(g_n)]\vert_{U_k}=F_k$ for every $k$. Let $K\Subset U$.  Let $n_1$ be such that $K\Subset V_{n_1}^{-\vare_{n_1}}$ and let $n_2$ be such that $s(\phi_{n})<d(K,(V_{n_1}^{-\vare_{n_1}})^\complement)$ for all $n>n_2$.  If $W=V_{n_1}^{-\vare_{n_1}}$ and $n_0=\max\{n_1,n_2\}$, then by Lemma \ref{commute} and the construction of $G_n$, for each $n>n_0$ we have
$$
g_n\vert_W=(G_n\ast\phi_n)\vert_W=G_n\vert_{V_{n_1}}\ast\phi_n=G_{n_1}\ast\phi_n
$$
on $K$.  By Lemma \ref{contex}, $\left(g_n\vert_W\right)$ is fundamental on $V_{n_1}$, so there exists $(\psi_n)\in\Delta$ such that $g_n\vert_{V_{n_2}^{-\vare_{n_2}}}\ast\psi_m$ converges uniformly on $K$ for each $m$.  Since $K$ is an arbitrary compact subset of $U$, this shows that $(g_n)$ is fundamental on $U$.

To show that  $[(g_n)]\vert_{U_k}=F_k$ for every $k$, we will show that for fixed $k$, some $(f_n)\in F_k$ and any $K\Subset U_{k}$, there is $(\psi_n)\in\Delta$ such that $(g_n\vert_{U_k}-f_n)\ast\psi_m \toto 0$ for each fixed $m$ .  We choose $n_0$ such that $g_n=G_n\ast \phi_n$ on $U_k$ and that $s(\phi_n)<d(K,(U_k)^\complement)$ for every $n\geq n_0$.  By Lemma \ref{commute} and the fact that $G_k\vert_{U_k}=F_k$, we have $\left( G_n\ast\phi_n \right)\vert_{U_k^{-\vare_{n_0}}}=F_k\ast\phi_n$ on $U_k^{-\vare_{n_0}}$. Hence, for each $n$,
$$
g_n\vert_{U_k^{-\vare_{n}}}=\left( G_n\ast\phi_n \right)\vert_{U_k^{-\vare_{n}}} = F_k\ast\phi_n
$$
on $U_k^{-\vare_{n_0}}$. Let, for every $n$, $h_n$ be a continuous extension of $g_n\vert_{U_k^{-\vare_{n}}}$ to $U_k$. Again by Lemma \ref{contex}, $h_n= F_k$.  Since $h_n=g_n$ on $K$ for $n>n_0$, for any $(\psi_n)\in\Delta$, $(h_n-g_n)\ast\psi_m\toto 0$ on $K$ for every $m$.

\end{proof}

\begin{thm}
The space of Boehmians over $\rn$ is a sheaf.
\end{thm}

 \begin{proof}
It suffices to show properties (a) and (b).  Property (a) follows from Lemma \ref{prop1}. For the second property, we assume that $\{U_\alpha\}_{\alpha\in I}$ is a open cover of $U$. Let $\{U_{\alpha_n}\}$ be a countable subcover of $\{U_\alpha\}$.  Then, by Lemma \ref{countable}, there is a Boehmian $F\in\boehms{U}$ such that $F\vert_{U_{\alpha_n}}=F_{\alpha_n}$ for every $n$.  We must show that $F\vert_{U_\alpha}=F_\alpha$ for every $\alpha$. Let $\alpha \in I$ and let  $V_n=U_\alpha\cap U_{\alpha_n}$.   Without loss of generality we may assume that $V_n\neq\emptyset$.  Then $F_\alpha\vert_{V_n}=F_n\vert_{V_n}$ for all $n$.  Since $F\vert_{U_{\alpha_n}}=F_{\alpha_n}$, we have $F\vert_{V_n}=F_n\vert_{V_n}=F_\alpha\vert_{V_n}$ for all $n$.  By (a), we conclude $F\vert_{U_\alpha}=F_\alpha$.
\end{proof}

\section{Fundamental and Cauchy sequences} \label{Ch6}

In this paper we have used a space of Boehmians constructed by taking equivalence classes of fundamental sequences. A similar construction is used in previous papers by P. Mikusi\'nski which uses equivalence classes of Cauchy sequences. This difference has some technical consequences, but we will show that the constructed spaces of Boehmians are isomorphic.

Boehmians on open subsets of $\rn$ were introduced in \cite{open} with the idea of $\Delta$-convergence.

\begin{defn}[$\Delta$-Convergence]
A sequence of functions $f_n\in\cfuncs{U}$ is $\Delta$-convergent to $f\in\cfuncs{U}$, denoted by $f_n \Dto f$, if for each $K\Subset U$ there exists $(\delta_n)\in\Delta$ such that  $(f-f_n)\ast\delta_n\toto 0$ on $K$ as $n\to \infty$.  \end{defn}

It is then shown that $\cfuncs{U}$, endowed with $\Delta$-convergence, is a linear metric space.  The space of Boehmians on $U$ is defined as the completion of $\cfuncs{U}$ with respect to $\Delta$-convergence.  We will denote this space of Boehmians by $\oboehms{U}$. The equivalence class of a Cauchy sequence $(f_n)$ will be denoted $[(f_n)]_\Delta$.  Note, if $(f_n)$ and $(g_n)$ are such that for each $K\Subset U$ there is a $(\delta_n)$ so that $(f_n-g_n)\ast\delta_n\to 0$ as $n\to\infty$, then $(f_n)$ and $(g_n)$ are members of the same equivalence class in $\oboehms{U}$.

\begin{defn}[$\Delta$-Cauchy]
For an open set $U\subseteq\reals$ and a sequence of functions $f_n\in\cfuncs{U}$, we say that $(f_n)$ is $\Delta$-Cauchy if for all increasing sequences of natural numbers $(p_n)$ and $(q_n)$, and for every $K\Subset U$, there exists $(\varphi_n)\in\Delta$ such that $\Knorm{(f_{p_n}-f_{q_n})\ast\varphi_n}{K}\to 0$.
\end{defn}

In other words, $(f_n)$ is $\Delta$-Cauchy if $f_{p_n}-f_{q_n}\Dto 0$ for all increasing sequences of natural numbers $(p_n)$ and $(q_n)$. It is clear that if a sequence of functions in $\cfuncs{U}$ is $\Delta$-Cauchy then it is a member of some Boehmian in $\oboehms{U}$.

A second type of convergence considered in \cite{open} is $\delta$-convergence.

\begin{defn}[$\delta$-convergence]
A sequence of functions $f_n\in\cfuncs{U}$ is $\delta$-convergent to zero, denoted by $f_n\dto 0$, if for each $K\Subset U$ there exists $(\delta_n)\in\Delta$ such that $\Knorm{f_n\ast\delta_m}{K}\to 0$ as $n\to\infty$ for each $m\in\nats$.  If $(f_n-f)\dto 0$, we write $f_n\dto f$.
\end{defn}

\begin{defn}[$\delta$-Cauchy]
 A sequence of functions $f_n\in\cfuncs{U}$ is called a $\Delta$-Cauchy sequence if for all increasing sequences of natural numbers $(p_n)$ and $(q_n)$ we have $f_{p_n}-f_{q_n}\dto 0$.
\end{defn}

 	Note that two fundamental sequences $(f_n)$ and $(g_n)$ are equivalent if $f_n-g_n\dto 0$ and a sequence $(f_n)$ is \textit{$\delta$-Cauchy} if and only if it is a fundamental sequence.  We see that our construction of Boehmians as $\mathscr{A}/\sim$ is the completion of the space endowed with $\delta$-convergence. Following are some important lemmas from \cite{open} about the relationship between $\delta$-convergence and $\Delta$-convergence.

\begin{lemm}\label{dD}
$f_n\dto f$ implies $f_n\Dto f$.
\end{lemm}
\begin{lemm}\label{subseq}
If $f_n\Dto f$, then there exists a subsequence $(f_{p_n})$ of $(f_n)$ such that $f_{p_n}\dto f$.
\end{lemm}

First we show that there is a surjection from $\oboehms{U}$ to $\boehms{U}$.

\begin{cor}\label{surjective}
For every equivalence class $[(f_n)]_\delta\in\boehms{U}$, there is a class $[(g_n)]_\Delta\in\oboehms{U}$ such that $[(f_n)]_\delta\subseteq[(g_n)]_\Delta$.
\end{cor}
\begin{proof}
Consider $[(f_n)]_\delta\in\boehms{U}$.  Since every element of $[(f_n)]_\delta$ is a fundamental sequence, for any increasing sequences of natural numbers $(p_n)$ and $(q_n)$ we have $f_{p_n}-f_{q_n}\dto 0$.  By Lemma \ref{dD}, $f_{p_n}-f_{q_n}\Dto 0$.  Thus, $(f_n)$ is $\Delta$-Cauchy, and so $[(f_n)]_\Delta$ exists.  Additionally, for any other $(g_n)\in[(f_n)]_\delta$, we have that $f_n-g_n\dto 0$, so $f_n-g_n\Dto 0$.  Thus, $(g_n)\in [(f_n)]_\Delta$.
\end{proof}

Now we show that each class of $\Delta$-sequences contains a unique class of $\delta$-sequences. Let $(K_j)$ be a sequence of compact subsets of an open subset $U\subseteq\reals^N$ such that
\begin{itemize}
\item[(i)]$\bigcup_{j=1}^{\infty}K_j=U,$
\item[(ii)]$K_j\subset K_{j+1}~\mathrm{for}~ j\in\nats,$
\item[(iii)]$d(K_{j+1},U^\complement)<\frac{1}{2}d(K_{j},U^\complement)~\mathrm{for}~j\in\nats$ and
\item[(iv)]$d(K_{1},U^\complement)<\frac{1}{2}.$
\end{itemize}
For a continuous function $f$ on $U$, we define $$P_j(f)=\inf\{\|f\ast\phi\|_{K_j}:\phi\in\mathscr{D}_o~\mathrm{and}~s(\phi)<d(K_{j},U^\complement)\}.$$

\begin{thm}\label{equiv}
Let $U\in\reals^N$ be an open set and let $f_n\in\mathscr{C}(U)$ for $n\in\nats$.  Then $f_n\overset{\Delta}\longrightarrow 0$ if and only if for each $j\in\nats$, $P_j(f_n)\to0$ as $n\to\infty$.
\end{thm}

\begin{proof}
See \cite{open}.
\end{proof}

\begin{lemm}\label{seqconstr}
Let $U\subseteq\reals^N$ be an open set and let $(K_j)$ be a sequence of compact subset of $U$ satisfying conditions (i)-(iv) above.  If $(f_n)$ is $\Delta$-Cauchy, then there is a subsequence c of $(f_n)$, such that for every $n$ there are $\vare_n\geq 0$ and $\phi_n\in\mathscr{D}_o$ such that
\begin{itemize}
\item[(1)]  $\sum_{i=1}^\infty\vare_i<\infty,$
\item[(2)]  $s(\phi_n)<\vare_{n},$
\item[(3)]  $(f_{p_{n+1}}-f_{p_n})\ast\phi_n\in\mathscr{C}(K_{n}),$
\item[(4)] $\|(f_{p_{n+1}}-f_{p_n})\ast\phi_n\|_{K_{n+1}}<\frac{1}{2^{n}}.$
\end{itemize}
Moreover, the sequence $(f_{p_n})$ is fundamental.
\end{lemm}

\begin{proof}
Let $(\vare_j)$ be a sequence such that $\vare_j<d(K_{j},U^\complement)$ for every $j$.  By $(iii)$ and $(iv)$ above, $\sum_{i=1}^\infty\vare_i<\sum_{i=1}^{\infty}\frac{1}{2^i}<\infty$. Thus the sequence $(\vare_j)$ meets the first condition. Since $(f_n)$ is $\Delta$-Cauchy, we have from Theorem \ref{equiv} that for every $k$, $P_k(f_n-f_m)\to 0$ as $m,n\to\infty$. Thus, for $\vare_1$ there exists $p_1\in\nats$ such that for every $m,n\geq p_1$, $P_1(f_n-f_m)<\vare_1$.  This implies that there is a function $\phi_1\in\mathscr{D}_o$ with $s(\phi_1)<\vare_1$ and $\|(f_n-f_m)\ast\phi_1\|_{K_2}< \vare_1<\frac{1}{2}$ for every $m,n\geq p_1$.   As we have $P_k(f_n-f_m)\to 0$ for every $k$, we can find $p_k\in\nats$ in a similar fashion for every $k$.  That is, we can find a $p_k>p_{k-1}$ such that there is a $\phi_k\in\mathscr{D}_o$ with $s(\phi_k)<\vare_{k}$ and $\|(f_n-f_m)\ast\phi_k\|_{K_{k}}\leq\vare_{k}$ for every $m,n\geq p_k$.  By this process, for every $n$ we select $\phi_n$ such that $s(\phi_n)<\vare_{n}$, meeting the second condition.  For the third, note that for any $n$, $f_{p_n+1}-f_{p_n}$ is a continuous function defined on $U$  and $s(\phi_n)<\vare_{n}<d(K_{n},U^\complement)$, so $(f_{p_{n+1}}-f_{p_n})\ast\phi_n$ is defined on $K_{n}$ and is continuous. Note that $\vare_k<\frac{\vare_{k-1}}{2}$ and $\vare_1<\frac{1}{2}$, so $\vare_k<\frac{1}{2^k}$. Thus, since $\|(f_n-f_m)\ast\phi_k\|_{K_k}\leq\vare_{k}$ for every $m,n\geq p_k$, we have $\|(f_{p_{n+1}}-f_{p_n})\ast\phi_n\|_{K_n}<\frac{1}{2^{n}}$ for each $n$. Note that the sequence $(\phi_n)$ as defined above is a delta sequence since $s(\phi_n)<\vare_{n}\to 0$.

Now we show that the sequence $(f_{p_n})$ is fundamental.
Let $K\Subset U$.  Then $K\subseteq K_j$ for some  $j$.  Let $\delta_n=\phi_n\ast\phi_{n+1}\ast\ldots$ and, for $l>n$,
$$
\psi_n^{l}=\phi_n\ast\phi_{n+1}\ast\ldots\ast\phi_{l-1}\ast\phi_{l+1}\ast\ldots.
$$
From \cite{open} we know that $(\delta_n),(\psi_n^l)\in\Delta$, and $s(\psi_n^l)<s(\delta_n)<\frac{1}{2^{n-1}}$ for any $l$ (since $s(\phi_n)<\frac{1}{2^{n}}$ and $s(\delta_n)\leq\sum_{i=n}^\infty\phi_i$).  We claim that for every fixed $k$, $(f_{p_n}\ast\delta_k)$ converges uniformly on $K_j$ as $n\to\infty$.  As $(f_{p_n}\ast\delta_k)$ is a sequence of continuous functions on $K_j$ for large enough $k$, this is equivalent to showing that $(f_{p_n}\ast\delta_k)$ is Cauchy on $K_j$ with respect to uniform convergence.  This is done as follows, where we assume, without loss of generality that $j<k<m<n$ and $k$ is large enough so that $(f_{p_n}-f_{p_m})\ast\delta_k$ is defined on $K_j$:
\begin{align*}
\|(f_{p_n}-f_{p_m})\ast\delta_k\|_{K_j}&=\left\|\sum_{i=m+1}^n(f_{p_i}-f_{p_{i-1}})\ast\delta_k\right\|_{K_j}\\
&\leq \sum_{i=m+1}^n\left\|(f_{p_i}-f_{p_{i-1}})\ast\delta_k\right\|_{K_j}\\
&=\sum_{i=m+1}^n\left\|(f_{p_i}-f_{p_{i-1}})\ast\phi_{i-1}\ast\psi_k^{i-1}\right\|_{K_j}\\
&\leq\sum_{i=m+1}^n\left\|(f_{p_i}-f_{p_{i-1}})\ast\phi_{i-1}\right\|_{K_{i-1}}\left\|\psi_k^{i-1}\right\|_1\\
&=\sum_{i=m+1}^n\left\|(f_{p_i}-f_{p_{i-1}})\ast\phi_{i-1}\right\|_{K_{i-1}}\\
& < \sum_{i=m+1}^n \frac1{2^{i-1}} \to 0.
\end{align*}
\end{proof}

\begin{cor}\label{injective}
For each $[(g_n)]_\Delta\in\oboehms{U}$, there is a unique $[(f_n)]_\delta\in\boehms{U}$ such that $[(f_n)]_\delta\subseteq[(g_n)]_\Delta$.
\end{cor}

\begin{proof}
By Lemma \ref{seqconstr}, for every $(f_n)\in [(g_n)]_\Delta$, there is a subsequence $(f_{p_n})$ such that $(f_{p_n})$ is fundamental. By Lemma \ref{dD}, $[(f_{p_n})]_\delta\subset[(g_n)]_\Delta$. To show uniqueness, suppose $[(f_n)]_\delta,[(h_n)]_\delta\subseteq [(g_n)]_\Delta$. Then $(f_n-h_n)\Dto 0$ and there are subsequences $(f_{q_n})$ and $(h_{q_n})$ of $(f_n)$ and $(g_n)$, respectively, such that $(f_{q_n}-h_{q_n})\dto 0$.  But then $(h_{q_n})\in [(f_n)]_\delta$. Now, since $(f_{q_n}-f_n)\dto 0$ and $(h_{q_n}-h_n)\dto 0$, we have $(h_n)\in[(f_n)]_\delta$. Hence $[(f_n)]_\delta=[(h_n)]_\delta$.
\end{proof}

Corollaries \ref{surjective} and \ref{injective} imply that there is a one-to-one correspondence between equivalence classes of fundamental sequences and equivalence classes of $\Delta$-Cauchy sequences.


\begin{thebibliography}{99}
\bibitem{jan} P. Antonsik and A. Kami{\'n}ski: \emph{Generalized Functions and Convergence}, 1990, World Scientific.
\bibitem{mikops} H. Bowman and M.A. Goldberg: \emph{Mikusi{\'n}ski Operators Without the {T}itchmarsh Theorem}, Amer. Math. Monthly, \textbf{87}, (1980), 564-567.
\bibitem{supp} T.K. Boehme: \emph{The Support of {M}ikusi{\'n}ski Operators}, Trans. Amer. Math Soc., \textbf{176}, (1973), 319-334.
\bibitem{dists} J. L{\"u}tzen: \emph{The Prehistory of the Theory of Distributions}, 1982, Springer-Verlag.
\bibitem{open} P. Mikusi\'nski: \emph{Boehmians on Open Sets}, Acta Math. Hungar., \textbf{55(1-2)}, (1990), 63-73.
\bibitem{deltas} P. Mikusi\'nski: \emph{On Delta Sequences}, Estratto da Rend. Ist. Matem. Univ. Trieste, \textbf{19}, (1987), 165-175.
\bibitem{conv} P. Mikusi\'nski: \emph{Convergence of Boehmians}, Japan J. Math., \textbf{9}, (1982), 159-179.
\bibitem{flex} P. Mikusi\'nski: \emph{On Flexibility of Boehmians}, Proceedings of the Conference Different Aspects of Differentiability II (Warsaw, 1995), Integral Transform. Spec. Funct. \textbf{4}, (1996), 141-146.
\bibitem{sheaf} J. A. Seebach, et al.: \emph{What is a Sheaf?}, The American Mathematical Monthly, \textbf{77(7)} (Aug. - Sep., 1970), 681-703.
\end{thebibliography}
\end{document}